\documentclass[oneside,12pt]{amsart}
\usepackage[T2A]{fontenc}
\usepackage[cp1251]{inputenc}
\usepackage{amssymb}
\usepackage{amsxtra}
\usepackage{amsmath}
\usepackage{amstext,amsthm,amsfonts,amscd}

\usepackage[all]{xy}

\usepackage{hyperref}
\usepackage{tikz,graphicx}
\usetikzlibrary{matrix,arrows}

\usepackage[russian,english]{babel}

\DeclareMathOperator{\GL}{GL}
\DeclareMathOperator{\PGL}{PGL}

\DeclareMathOperator{\Cent}{Cent}

\DeclareMathOperator{\Spec}{Spec}

\DeclareMathOperator{\Gm}{{\mathbf G}_m}

\DeclareMathOperator{\Aut}{Aut}

\DeclareMathOperator{\Gal}{Gal}

\newcommand{\Aff}{\mathbb {A}}
\newcommand{\Pro}{\mathbb {P}}

\DeclareMathOperator{\ZZ}{{\mathbb Z}}

\DeclareMathOperator{\et}{\text{\it \'et}}

\newtheorem{lem}{Lemma}[section]
\newtheorem*{lem*}{Lemma}
\newtheorem*{thm*}{Theorem}
\newtheorem{thm}[lem]{Theorem}
\newtheorem{cor}[lem]{Corollary}

\theoremstyle{definition}{    }
\theoremstyle{definition}{   }
\theoremstyle{definition}{  \newtheorem{defn}[lem]{Definition} }
\theoremstyle{definition}{  \newtheorem{conj}[lem]{Conjecture} }

\let\l\left
\let\r\right

\let\semil\rtimes

\newcommand{\st}{\scriptstyle}

\makeatletter

\@addtoreset{equation}{section}

\makeatother

\begin{document}

\selectlanguage{english}

\title{Torsors of isotropic reductive groups over Laurent polynomials}

\author{Anastasia Stavrova}
\thanks{
The author is a winner of the contest ``Young Russian Mathematics''. The work was supported
by the Russian Science Foundation grant 19-71-30002.}
\address{Chebyshev Laboratory, St. Petersburg State University,
St. Petersburg, Russia}
\email{anastasia.stavrova@gmail.com}
\subjclass[2010]{14F20, 20G35, 17B67, 19B28, 11E72}
\keywords{isotropic reductive group, loop reductive group, Laurent polynomials, $G$-torsor, non-stable $K_1$-functor,
Whitehead group}

\maketitle

\begin{abstract}
Let $k$ be a field of characteristic 0. Let $G$ be a reductive group over the ring of Laurent polynomials
$R=k[x_1^{\pm 1},...,x_n^{\pm 1}]$. We prove that $G$ has isotropic rank $\ge 1$ over $R$ iff it has isotropic rank $\ge 1$
over the field of fractions $k(x_1,\ldots,x_n)$ of $R$, and if this is the case, then
the natural map $H^1_{\et}(R,G)\to H^1_{\et}(k(x_1,\ldots,x_n),G)$ has trivial kernel, and $G$ is loop reductive, i.e.
contains a maximal $R$-torus.
In particular, we settle in positive the conjecture
of V. Chernousov, P. Gille, and A. Pianzola that $H^1_{Zar}(R,G)=*$ for such groups $G$.
We also deduce that if $G$ is a reductive group over $R$ of isotropic rank $\ge 2$, then
the natural map of non-stable $K_1$-functors
$K_1^G(R)\to K_1^G\bigl( k((x_1))...((x_n)) \bigr)$ is injective, and an isomorphism if $G$ is moreover semisimple.

\end{abstract}

\section{Introduction}

Let $k$ be a field of characteristic 0. Let $R=k[x_1^{\pm 1},...,x_n^{\pm 1}]$ be the
ring of Laurent polynomials over $k$.
Let $G$ be a reductive group scheme over $R$ in the sense of~\cite{SGA3}.
V. Chernousov, P. Gille, and A. Pianzola~\cite[Theorem 1.1]{ChGP-class} showed that classifying \'etale-locally trivial $G$-torsors over $R$ is equivalent
to classifying Zariski-locally trivial torsors over $R$ for all twisted $R$-forms of $G$ that are loop reductive,
i.e. contain a maximal $R$-torus. With respect to the latter problem, they proposed the following
conjecture.

\begin{conj}\cite[Conjecture 5.4]{ChGP-class}\label{conj:1}
Let $k$ be a field of characteristic 0. Let $G$ be a loop reductive group over the ring of Laurent polynomials
$R=k[x_1^{\pm 1},...,x_n^{\pm 1}]$. Assume that all semisimple quotients of $G$ are
isotropic, i.e. contain $\Gm_{,R}$. Then $H^1_{Zar}(R,G)$ is trivial.
\end{conj}

 If $G=\GL_n$, the conjecture amounts to the fact that all finitely generated projective modules
over $k[x_1^{\pm 1},...,x_n^{\pm 1}]$ are free, and it was established (in arbitrary characteristic) by R. G. Swan~\cite{Swan-Laurent}
relying on D. Quillen's proof of Serre's conjecture~\cite{Q}. More generally, if $G$ is defined over $k$,
the conjecture was proved by P. Gille and A. Pianzola~\cite{GiPi08} relying on a theorem of
M.S. Raghunathan on the triviality of $k[x_1,\ldots,x_n]$-torsors~\cite{Rag89}. Apart from that,
the conjecture was previously known for several classes of groups if $k$ is algebraically closed and $n=2$ (see
Corollary~\ref{cor:2} below);
for some twisted forms of $\GL_n$~\cite{Artamonov-crossed,ChGP-class} and of orthogonal groups~\cite{Par-Laurent}.
The isotropy condition in the statement is necessary,
since there are anisotropic forms of $\PGL_n$ over $k[x_1^{\pm 1},x_2^{\pm 2}]$ with non-trivial Zariski
cohomology~\cite[Corollary 3.22]{GiPi07}.

We establish the above conjecture in full by proving the following more general statement.

\begin{thm}\label{thm:sc}
Let $k$ be a  field of characteristic $0$, and let $G$ be a reductive group
over $R= k[x_1^{\pm 1},\ldots,x_n^{\pm 1}]$. Set $K=k(x_1,\ldots,x_n)$ and $F=k((x_1))\ldots((x_n))$.
\begin{enumerate}
\item\label{item:equiv} The following conditions on $G$ are equivalent:
\begin{enumerate}
\item $G$ has isotropic rank $\ge 1$ over $R$;
\item $G_K$ has isotropic rank $\ge 1$ over $K$;
\item $G_F$ has isotropic rank $\ge 1$ over $F$.
\end{enumerate}
\item If $G$ satisfies the equivalent conditions of~\eqref{item:equiv}, then $G$ is loop reductive
and for any regular ring $A$ containing $k$,  the natural map
$$
H^1_{\et}\bigl(k[x_1^{\pm 1},\ldots,x_n^{\pm 1}]\otimes_k A,G\bigr)\to H^1_{\et}\bigl(k(x_1,\ldots,x_n)\otimes_k A,G\bigr)
$$
has trivial kernel.
\end{enumerate}
\end{thm}

\begin{cor}
Let $k$ be a  field of characteristic $0$, and let $G$ be a reductive group of isotropic rank $\ge 1$ over
$R= k[x_1^{\pm 1},\ldots,x_n^{\pm 1}]$. Then $H^1_{Zar}\bigl(R,G\bigr)=1$.
\end{cor}

The proof of Theorem~\ref{thm:sc} relies on the ``diagonal argument'' trick for loop reductive groups~\cite{St-serr},
on the established cases of the Serre--Grothendieck conjecture~\cite{PaStV,PaF}, and on the
classification results~\cite[Theorem 1.2]{ChGP-class}.

Using the above results and a known
case of Serre's conjecture II~\cite{CTGP},  we also establish another conjecture of P. Gille and
A. Pianzola.

\begin{cor}\cite[Conjecture 6.1]{GiPi07}\label{cor:2}
Let $k$ be an algebraically closed field of characteristic $0$. Let
$G$ be a semisimple reductive group over $R=k[x_1^{\pm 1},x_2^{\pm 1}]$ having no semisimple normal subgroups
of type $A_n$, $n\ge 1$. Let $1\to \mu\to G^{sc}\to G\to 1$ be the simply connected cover of $G$.
Then the boundary map $H_{\et}^1(R,G)\to H^2_{\et} (R,\mu)$ is bijective. In
particular, if $G$ is simply connected, then $H^1_{\et}(R,G)$ is trivial.
\end{cor}

In~\cite{GiPi07} P. Gille and A. Pianzola established this conjecture for groups of types $G_2$, $F_4$ and $E_8$.
The case of groups of type $B_n$, $n\ge 2$, and of some groups of type $D_n$, follows from~\cite{Par-Laurent}. The groups of types $C_n$, $n\ge 6$,
and $D_n$, $n\ge 8$, were covered in~\cite{Steinmetz}. Our proof covers
all cases except for $E_8$, where we refer to~\cite{GiPi07}.

As another corollary, we remove the assumption of loop reductivity in a previous result of the author
concerning non-stable $K_1$-functors of isotropic reductive groups.
For any commutative ring $R$, once a reductive group $G$ over $R$ has isotropic rank $\ge 1$, then it contains
a pair of opposite strictly proper parabolic $R$-subgroups $P$ and $P^-$~\cite[Exp. XXVI]{SGA3}.
Under this assumption
one can consider the following "large"{} subgroup of $G(R)$ generated by unipotent elements,
$E_P(R)=\l<U_P(R),U_{P^-}(R)\r>$
where $U_P$ and $U_{P^-}$ are the unipotent radicals of $P$ and $P^-$.
The set of (left) cosets
$$
G(R)/E_P(R)=K_1^{G,P}(R)
$$
is called the non-stable $K_1$-functor associated to $G$ (and $P$), or
the Whitehead group of $G$. If $G$ has isotropic
rank $\ge 2$, that is, all semisimple quotients of $G$
contain $(\Gm_{,R})^2$, then $K_1^{G,P}(R)$ is independent of the choice of $P^\pm$~\cite{PS} and we denote it by $K_1^G(R)$.
The following result is a combination of~\cite[Theorem 1.2]{St-serr} and Theorem~\ref{thm:sc}. Its surjectivity part
follows from~\cite[Theorem 14.3]{ChGP-conj}.

\begin{cor}\label{cor:K1}
Let $k$ be a  field of characteristic $0$, and let $G$ be a reductive group of isotropic rank $\ge 2$ over
$R= k[x_1^{\pm 1},\ldots,x_n^{\pm 1}]$. Then the natural map
$$K_1^G(R)\to K_1^G\bigl( k((x_1))...((x_n)) \bigr)$$ is injective.  If $G$ is moreover semisimple, the map is an
isomorphism.
\end{cor}

The author heartily thanks Vladimir Chernousov and Philippe Gille for
illuminating discussions and kind attention to the present work.

\section{Preliminaries on loop reductive groups}

Let $k$ be a field of characteristic $0$. We fix once and for
all an algebraic closure $\bar k$ of $k$ and a compatible set of primitive $m$-th roots of unity
$\xi_m\in \bar k$, $m\ge 1$.

P. Gille and A. Pianzola~\cite[Ch. 2, 2.3]{GiPi-mem} compute the \'etale
(or algebraic) fundamental group of the $k$-scheme
$$
X=\Spec k[x_1^{\pm 1},\ldots, x_n^{\pm 1}]
$$ at the
natural geometric point $e:\Spec\bar k\to X$ induced by the evaluation $x_1=x_2=\ldots=x_n=1$. Namely,
\begin{equation}\label{eq:pi1X}
\pi_1(X,e)=\hat\ZZ(1)^n\semil\Gal(\bar k/k),
\end{equation}
where $\hat\ZZ(1)$ denotes the profinite group $\varprojlim\limits_{m}\mu_m(\bar k)$ equipped with the
natural action of the absolute Galois group $\Gal(\bar k/k)$.

For any reductive group scheme $G$ over $X$, we denote by $G_0$ the split, or Chevalley---Demazure
reductive group in the sense of~\cite{SGA3} of the same type as $G$. The group $G$ is a twisted form of $G_0$, corresponding
to a cocycle class $\xi$ in the \'etale cohomology set $H^1_{\et}(X,\Aut(G_0))$.
The group $\pi_1(X,e)$ acts continuously on $\Aut(G_0)(\bar k)$ via the natural homomorphism
$\pi_1(X,e)\to \Gal(\bar k/k)$.

\begin{defn}\label{def:loop}\cite[Definition 3.4]{GiPi-mem}
The group scheme $G$ is called \emph{loop reductive}, if the cocycle $\xi$ is in the image of the natural map
$$
H^1\bigl(\pi_1(X,e),\Aut(G_0)(\bar k)\bigr)\to H^1_{\et}\bigl(X,\Aut(G_0)\bigr),
$$
where $H^1\bigl(\pi_1(X,e),\Aut(G_0)(\bar k)\bigr)$ denotes the non-abelian cohomology set
in the sense of Serre~\cite{Se}.
\end{defn}

This definition can be reformulated as follows.

\begin{thm*}\cite[Corollary 6.3]{GiPi-mem}
A reductive group scheme over $X$
is loop reductive if and only if $G$ has a maximal torus over $X$.
\end{thm*}

The definition of a maximal torus is as follows.

\begin{defn}\cite[Exp. XII D\'ef. 3.1]{SGA3}
Let $G$ be a group scheme of finite type over a scheme $S$, and let $T$ be a $S$-torus
which is an $S$-subgroup scheme of $G$.
Then $T$ is a \emph{maximal torus of $G$ over $S$}, if $T_{\overline{k(s)}}$ is a maximal torus of
$G_{\overline{k(s)}}$ for all $s\in S$.
\end{defn}

We say that $G$ has isotropic rank $\ge n$ if every normal semisimple reductive $R$-subgroup of $G$ contains
$(\Gm_{,R})^n$. This is equivalent to saying that all semisimple quotients of $G$ contain $(\Gm_{,R})^n$.

\section{Some corollaries of the Serre--Grothendieck conjecture for isotropic groups}

The following statement was obtained in~\cite{St-dh} as a joint corollary
of the corresponding statement for simply connected semisimple reductive groups~\cite[Theorem 1.6]{PaStV},
and of the result of I. Panin and R. Fedorov on the Serre--Grothendieck conjecture~\cite{PaF}.

\begin{thm*}~\cite[Theorem 4.2]{St-dh}
Assume that $R$ is a regular semilocal domain that contains an infinite field, and let $K$ be its fraction field. Let
$G$ be a reductive group scheme over $R$
of isotropic rank $\ge 1$. Then for any $n\ge 1$ the natural map
$$
H^1_{\et}(R[x_1,\ldots,x_n],G)\to H^1_{\et}(K[x_1,\ldots,x_n],G)
$$
has trivial kernel.
\end{thm*}

\begin{lem}\label{lem:regloc-inj}
Let $G$ be a reductive group of isotropic rank $\ge 1$ over a regular local ring $A$ containing an infinite field
$k$. Let $f(x)\in A[x]$ be a non-zero polynomial.
Then $H^1_{\et}(\Aff^1_A,G)\to H^1_{\et}((\Aff^1_A)_f,G)$
has trivial kernel.
\end{lem}
\begin{proof}
Let $K$ be the fraction field of $A$. By~\cite[Theorem 4.2]{St-dh} the map $H^1_{\et}(A[x],G)\to H^1_{\et}(K[x],G)$
has trivial kernel. By~\cite[Proposition 2.2]{CTO} the map $H^1_{\et}(K[x],G)\to H^1_{\et}(K(x),G)$ has trivial kernel.
Hence the claim.
\end{proof}

The following lemma is based on a classical trick of Quillen~\cite{Q}.

\begin{lem}\label{lem:quillen}
Let $G$ be a reductive group of isotropic rank $\ge 1$ over a regular ring $A$ containing
an infinite field $k$.
Let $f(x)\in A[x]$ be a monic polynomial.
Then $H^1_{\et}(\Aff^1_A,G)\to H^1_{\et}((\Aff^1_A)_{f},G)$
has trivial kernel.
\end{lem}
\begin{proof}
Let $\xi\in H^1_{\et}(\Aff^1_A,G)$ be in the kernel.
Since $f$ is monic, for any maximal ideal $m$ of $A$ the
image of $f$ in $A_m[x]$ is non-zero. Then by Lemma~\ref{lem:regloc-inj} the
$G$-bundle $\xi|_{\Aff^1_{A_m}}$ is trivial. Since $A$ is regular, $G$ is $A$-linear by~\cite[Corollary 3.2]{Thomason}.
Then by~\cite[Theorem 3.2.5]{AHW} (see also~\cite[Korollar 3.5.2]{Mo})
the fact that for any maximal ideal $m$ of $A$ the $G$-bundle $\xi|_{\Aff^1_{A_m}}$ is trivial implies
that  $\xi$ is extended from $A$.

Set $y=x^{-1}$ and choose $g(y)\in A[y]$ so that $x^{\deg(f)}g(y)=f(x)$. Then $g(0)\in A^\times$ and  $A[x]_{xf}=A[y]_{yg}$.
We have $\Pro^1_A=\Aff^1_A\cup\Spec(A[y]_g)$, and $\Aff^1_A\cap\Spec(A[y]_g)=(\Aff^1_A)_{xf}$.
Hence we can extend $\xi$ to a bundle $\hat\xi$ on $\Pro^1_A$ by gluing it to a trivial bundle on $\Spec(A[y]_g)$.
Let $\eta=\hat\xi|_{\Spec(A[y])}$. By assumption, $\eta$ is trivial on $\Spec(A[y]_g)$. Since $g(0)\in A^\times$,
by the same argument as above
$\eta$  is extended. However, $g(0)$ is invertible and $\eta$ is trivial at $y=0$, hence $\eta$ is trivial.
Hence $\xi$ is trivial at $x=y=1$.
Hence $\xi$ is trivial.
\end{proof}

\begin{lem}\label{lem:x-quillen}
Let $G$ be a reductive group of isotropic rank $\ge 1$ over a regular ring $A$ containing
an infinite field $k$.
Let $f(x)\in A[x]$ be a monic polynomial such that $f(0)\in A^\times$.
Then $H^1_{\et}((\Aff^1_A)_x,G)\to H^1_{\et}((\Aff^1_A)_{xf},G)$
has trivial kernel.
\end{lem}
\begin{proof}
Since $f(0)\in A^\times$, any
$G$-bundle in the kernel can be extended to $\Aff^1_A$ by gluing it to a trivial $G$-bundle on $(\Aff^1_A)_f$.
Then it is trivial by Lemma~\ref{lem:quillen} applied to $xf$.
\end{proof}

\begin{lem}\label{lem:const-inj}
Under the assumptions of Lemma~\ref{lem:quillen} for any $n\ge 0$
the natural map
$$
H^1_{\et}\bigl(A[t_1^{\pm 1},\ldots,t_n^{\pm 1}],G\bigr)\to H^1_{\et}\bigl(A\otimes_k k(t_1,\ldots,t_n),G\bigr)
$$
has trivial kernel.
\end{lem}
\begin{proof}
We prove the claim by induction on $n$; the case $n=0$ is trivial.
Set $l=k(t_1,\ldots,t_{n-1})$. By the inductive hypothesis, the map
$$
H^1_{\et}\bigl(A[t_1^{\pm 1},\ldots,t_n^{\pm 1}],G\bigr)\to H^1_{\et}\bigl(A[t_n^{\pm 1}]\otimes_k l,G\bigr)=
H^1_{\et}\bigl(A\otimes_k l[t_n^{\pm 1}],G\bigr)
$$
has trivial kernel, so it remains to prove the triviality of the kernel for the map
$$
H^1_{\et}\bigl(A\otimes_k l[t_n^{\pm 1}],G\bigr)\to H^1_{\et}\bigl(A\otimes_k l(t_n),G\bigr).
$$
We have $l(t_n)=\varinjlim\limits_{g} l[t_n]_{t_ng}$, where $g\in l[t_n]$ runs over all monic polynomials
with $g(0)\in l^\times$. Since $H^1_{\et}(-,G)$ commutes with filtered direct limits, it remains to show that every map
\begin{equation}\label{eq:coinj}
H^1_{\et}(A\otimes_k l[t_n^{\pm 1}],G)\to H^1_{\et}(A\otimes_k l[t_n]_{t_ng},G)
\end{equation}
has trivial kernel. This is the claim of Lemma~\ref{lem:x-quillen}.
\end{proof}

\begin{lem}\label{lem:zwt}
Let $k$ be an infinite field, $A$ be a regular ring containing $k$, and let $G$ be
a reductive group of isotropic rank $\ge 1$ over
$A[z_1^{\pm 1},\ldots,z_n^{\pm 1}]$.
For any set of integers $d_i>0$, $1\le i\le n$, the map
$$
\psi:H^1_{\et}\bigl(A[z_1^{\pm 1},\ldots,z_n^{\pm 1},t_1,\ldots,t_n],G\bigr)\xrightarrow{z_i\mapsto w_it_i^{d_i}}
H^1_{\et}\bigl(A\otimes _k k(w_1,\ldots,w_n)[t_1^{\pm 1},\ldots,t_n^{\pm 1}],\psi^*(G)\bigr)
$$
has trivial kernel.
\end{lem}
\begin{proof}
We prove the claim by induction on $n\ge 0$. The case $n=0$ is trivial.
To simplify the notation, set
$$
B=A[z_{2}^{\pm 1},\ldots,z_n^{\pm 1},t_{2},\ldots,t_n]$$
and $z=z_1$, $t=t_1$, $w=w_1$. Let $\phi:B[z^{\pm 1},t]\to B\otimes_k k(w)[t^{\pm 1}]$ be the map sending
$z$ to $wt^d$.
To prove the induction step for $n\ge 1$, it is enough to show that the induced map of \'etale cohomology
$$
\phi:H^1_{\et}\bigl(B[z^{\pm 1},t],G\bigr)\xrightarrow{z\mapsto wt^d} H^1_{\et}\bigl(B\otimes_k k(w)[t^{\pm 1}],\phi^*(G)\bigr)
$$
has trivial kernel, where $G$ is defined over $B[z^{\pm 1}]$.
Indeed, after that we can apply the induction assumption with $k$ substituted by $k(w_1)$ and $A$
substituted by $A\otimes_k k(w_1)[t_1^{\pm 1}]$.

We have
$$
B\otimes_k k(w)[t^{\pm 1}]=
\varinjlim\limits_{g} B\otimes_k k[w^{\pm 1}]_g[t^{\pm 1}]=\varinjlim\limits_{g} B\otimes_k k[w^{\pm 1},t^{\pm 1}]_g,
$$
where $g=g(w)$ runs over all monic polynomials in $k[w]$ with $g(0)\neq 0$. Let $N=\deg(g)\ge 1$. Since $\phi(z)=wt^d$, we have
$g(w)=g(\phi(z)t^{-d})=t^{-Nd}f(t)$, where $f(t)$ is a polynomial in $t$ with coefficients in
$k[\phi(z)^{\pm 1}]$ such that its leading coefficient is in $k\setminus 0$, and $f(0)=\phi(z)^N$. Then
$$
B\otimes_k k[w^{\pm 1},t^{\pm 1}]_g=B\otimes_k k[\phi(z)^{\pm 1},t]_{tf}.
$$
The group scheme $\phi^*(G)$ is defined over $B\otimes_k k[\phi(z)^{\pm 1}]$. Both terminal coefficients of $tf(t)$
are invertible in $k[\phi(z)^{\pm 1}]$, hence by Lemma~\ref{lem:quillen} applied to the regular ring
$B\otimes_k k[\phi(z)^{\pm 1}]$
the map
$$
H^1_{\et}\bigl(B[z^{\pm 1},t],G\bigr)\xrightarrow{z\mapsto wt^d} H^1_{\et}\bigl(B\otimes_k k[w^{\pm 1},t^{\pm 1}]_g,\phi^*(G)\bigr)=
H^1_{\et}\bigl(B\otimes_k k[\phi(z)^{\pm 1},t]_{tf},\phi^*(G)\bigr)
$$
has trivial kernel.

Since $H^1_{\et}(-,G)$ commutes with filtered direct limits, we conclude that $\phi$
has trivial kernel.

\end{proof}

\section{Proof of the main results}\label{sec:main}

To prove our main result Theorem~\ref{thm:sc}, we need the following property of loop reductive groups.

\begin{lem}["diagonal argument"]\cite[Lemma 4.1]{St-serr}\label{lem:diag}
Let $k$ be a field of characteristic $0$. Let $G$ be a loop reductive group scheme over $R=k[x_1^{\pm 1},\ldots,x_n^{\pm 1}]$.
For any integer $d>0$, denote by $f_{z,d}$ (respectively, $f_{w,d}$) the composition of $k$-homomorphisms
$$
R\to k[z_1^{\pm 1},\ldots,z_n^{\pm 1},w_1^{\pm 1},\ldots,w_n^{\pm 1}]
\to k[z_1^{\pm 1},\ldots,z_n^{\pm 1},(z_1w_1^{-1})^{\pm \frac1d},\ldots,(z_nw_n^{-1})^{\pm \frac1d}]
$$
sending $x_i$ to $z_i$ (respectively, to $w_i$) for every $1\le i\le n$. Then there is $d>0$ such that
$$
f_{z,d}^*(G)\cong f_{w,d}^*(G)
$$
as group schemes over $k[z_1^{\pm 1},\ldots,z_n^{\pm 1},(z_1w_1^{-1})^{\pm \frac1d},\ldots,(z_nw_n^{-1})^{\pm \frac1d}]$.
\end{lem}

\begin{lem}\label{lem:loop}
Let $k$ be a  field of characteristic $0$, and let $G$ be a loop reductive group of isotropic rank\ $\ge 1$
over $R= k[x_1^{\pm 1},\ldots,x_n^{\pm 1}]$. For any regular ring $A$ containing $k$,  the natural map
$$
H^1_{\et}\bigl(k[x_1^{\pm 1},\ldots,x_n^{\pm 1}]\otimes_k A,G\bigr)\to H^1_{\et}\bigl(k(x_1,\ldots,x_n)\otimes_k A,G\bigr)
$$
has trivial kernel.
\end{lem}
\begin{proof}

We apply Lemma~\ref{lem:diag} to $G$.
Set
$$
t_i=(z_iw_i^{-1})^{1/d},\quad 1\le i\le n,
$$
where $z_i,w_i$, and $d$ are as in that Lemma. Note that this is equivalent to
$$
z_i=w_it_i^d,\quad 1\le i\le n.
$$
We denote by $G_z$ the group scheme over
$k[z_1^{\pm 1},\ldots,z_n^{\pm 1}]$ which is the pull-back of $G$ under the $k$-isomorphism
$$k[x_1^{\pm 1},\ldots,x_n^{\pm 1}]\xrightarrow{x_i\mapsto z_i}k[z_1^{\pm 1},\ldots,z_n^{\pm 1}].$$
The group scheme $G_w$ over $k[w_1^{\pm 1},\ldots,w_n^{\pm 1}]$ is defined analogously. By Lemma~\ref{lem:diag}
$G_z$ and $G_w$ are isomorphic after pull-back to
$$
k[z_1^{\pm 1},\ldots,z_n^{\pm 1},t_1^{\pm 1},\ldots,t_n^{\pm 1}]=
k[w_1^{\pm 1},\ldots,w_n^{\pm 1},t_1^{\pm 1},\ldots,t_n^{\pm 1}].
$$

Consider the following commutative diagram. In this diagram, the horizontal maps $j_1$ and $j_2$ are the natural ones,
and all maps always take variables $t_i$ to $t_i$, $1\le i\le n$, and $A$ to $A$.
The bijections $g_1$ and $g_2$ exist by Lemma~\ref{lem:diag}.

\begin{tikzpicture}
\matrix (m) [matrix of math nodes,row sep=25pt,column sep=3em,minimum width=2em]
  {
   H^1_{\et}\Bigl(k[x_1^{\pm 1},\ldots,x_n^{\pm 1}]\otimes_k A,G\Bigr) & H^1_{\et}\Bigl(k(x_1,\ldots,x_n)\otimes_k A,G\Bigr) \\
    H^1_{\et}\Bigl(k[z_1^{\pm 1},\ldots,z_n^{\pm 1},t_1,\ldots,t_n]\otimes_k A,G_z\Bigr)& H^1_{\et}\Bigl(k(z_1,\ldots,z_n,t_1,\ldots,t_n)\otimes_k A,G_z\Bigr)   \\
    H^1_{\et}\Bigl(k(w_1,\ldots,w_n)[t_1^{\pm 1},\ldots,t_n^{\pm 1}]\otimes_k A,G_z\Bigr) &  \\
    H^1_{\et}\Bigl(k(w_1,\ldots,w_n)[t_1^{\pm 1},\ldots,t_n^{\pm 1}]\otimes_k A,G_w\Bigr) & H^1_{\et}\Bigl(k(w_1,\ldots,w_n,t_1,\ldots,t_n)\otimes_k A,G_w\Bigr)\\
  };
\path[-stealth]
     (m-1-1)       edge node [above] {$j_1$}  (m-1-2)
    (m-4-1.east) edge node [above] {$j_2$} (m-4-2)
;
\path[-stealth](m-1-2) edge node [left] {$\st f_2\colon x_i\mapsto z_i$}  (m-2-2);
\path [-stealth]
    (m-1-1)  edge node [left] {$\st f_1\colon x_i\mapsto z_i$}(m-2-1)
    (m-2-1) edge node [left] {$\st h\colon z_i\mapsto w_it_i^d$} (m-3-1)
    (m-3-1) edge node [left] {$\st g_1$} (m-4-1)
    (m-3-1) edge node [right] {$\cong$} (m-4-1)
    (m-2-2) edge node [left] {$\st g_2\colon  z_i\mapsto w_it_i^d$} (m-4-2)
    (m-2-2) edge node [right] {$\cong$}  (m-4-2);
;

\end{tikzpicture}

In order to prove that $j_1$ has trivial kernel, it is enough to show that all maps $j_2,g_1,h,f_1$ have trivial kernels.
The map $j_2$ has trivial kernel by Lemma~\ref{lem:const-inj}. As explained above, $g_1$ is bijective.
The map $h$ is has trivial kernel by Lemma~\ref{lem:zwt}. Finally, the map $f_1$ has trivial kernel,
since it has a retraction. Therefore, the map $j_1$ has trivial kernel.
\end{proof}

\begin{proof}[Proof of Theorem~\ref{thm:sc}]
To prove the first statement of the theorem, it is enough to show that if $G_F$ has isotropic rank $\ge 1$, then the same holds for $G$. Also,
we can assume from the start that $G$ is an adjoint reductive group over $R$. Then
$G$ is an inner twisted form of a uniquely determined quasi-split adjoint reductive $R$-group $G_{qs}$,
given by a cocycle class $\xi\in H^1_{\et}(S,\,G_{qs})$~\cite[Exp.~XXIV 3.12.1]{SGA3}. By definition, $G_{qs}$
contains a maximal $R$-torus, hence it is loop reductive.

By~\cite[Theorem 5.2]{ChGP-class} there is a
cocycle $\eta\in H^1_{\et}(R,G_{qs})$ such that the corresponding twisted group $H={}^\eta G_{qs}$ is also loop reductive,
and $\xi\in H^1_{Zar}(R,H)$. Then $G_K\cong H_K$ and
$G_F\cong H_F$. Since $H$ is loop reductive, by~\cite[Corollary 7.4]{GiPi-mem}
$H$ has isotropic rank $\ge 1$ if and only if $H_K$ has isotropic rank $\ge 1$ if and only if $H_F$ has isotropic rank $\ge 1$.
Thus, if $G_F$ has isotropic rank $\ge 1$, then $H$ has isotropic rank $\ge 1$ over $R$.
Then  by Lemma~\ref{lem:loop} we have $H^1_{Zar}(R,H)=1$. Then $G\cong H$. Consequently,
 $G$ has isotropic rank $\ge 1$ over $R$ and also $G$ is loop reductive.

To prove the second statement of the theorem, we note that the adjoint group $G^{ad}=G/\Cent(G)$ is loop reductive
by the above argument. Then $G$ is also loop reductive, since the maximal tori of $G$ and $G^{ad}$ are in bijective
correspondence by~\cite[Exp. XII 4.7.c]{SGA3}. Then the rest of the
second statement  holds by Lemma~\ref{lem:loop}.
\end{proof}

\begin{proof}[Proof of Corollary~\ref{cor:2}]
It was proved in~\cite[Theorem 3.17]{GiPi07} that boundary map
$$\delta_G:H_{\et}^1(R,G)\to H^2_{\et} (R,\mu)$$
induces
a bijection between $H^1_{loop}(R,G)$ and $H^1_{\et}(R,\mu)$, where $H^1_{loop}(R,G)\subset H^1_{\et}(R,G)$
is the subset of loop torsors, i.e. such $G$-torsors that the corresponging twisted form of $G$ is loop reductive.
In particular, the boundary map is surjective, and it remains to prove that it is injective.
Also, \cite[Theorem 2.7]{GiPi07} implies the conjecture for groups of pure type $E_8$.
If groups of this type occur as normal subgroups
in $G$, then they are necessarily direct factors, since they have trivial centers.
Hence we can assume that $G$ does not have semisimple normal subgroups of types $E_8$ or $A_n$, $n\ge 1$.

Set $K=k(x_1,x_2)$. Since $K$ has cohomological dimension 2, and for central simple algebras over a finite extension of
$K$ index coincides with exponent~\cite{dJ-perind}, the group $G_K$ is subject to~\cite[Theorems 1.2 and 2.1]{CTGP}. The latter
theorems imply that $G_K$ has isotropic rank $\ge 1$ over $K$, and that $H^1_{\et}(K,G^{sc})=1$.
In particular, $H_{\et}^1(K,G)\to H^2_{\et} (K,\mu)$ is bijective.

By Theorem~\ref{thm:sc} the fact that $G_K$ has isotropic rank $\ge 1$ implies that $G$  also has isotropic rank $\ge 1$
and is loop reductive.
Then $G^{sc}$ is loop reductive as well, and has isotropic rank $\ge 1$, since the maximal tori and parabolic subgroups
of $G$ and $G^{sc}$ are in bijective correspondence.
Then Theorem~\ref{thm:sc} applied to $G^{sc}$
implies that $H_{\et}^1(R,G^{sc})\to H^1_{\et} (K,G^{sc})$ has trivial kernel. Hence
$H_{\et}^1(R,G^{sc})$ is trivial and $\delta_G$ has trivial kernel.
Since all fibers of $\delta_G$ are in bijective correspondence with
kernels of $\delta_{G'}$ for suitable twisted forms $G'$ of $G$, we conclude that $\delta_G$ is injective.
\end{proof}

\begin{proof}[Proof of Corollary~\ref{cor:K1}]
Under the additional assumption that $G$ is loop reductive, the claim holds by~\cite[Theorem 1.2]{St-serr}. This assumption
is made redundant by Theorem~\ref{thm:sc}.

\end{proof}

\renewcommand{\refname}{References}

\end{document}